\documentclass[a4paper,11pt,oneside,reqno]{amsart}

\usepackage[english]{babel}
\usepackage[utf8x]{inputenc}    
\usepackage{xcolor}            
\usepackage{xspace}
\usepackage[pdftex]{graphicx}
\usepackage{epstopdf}
\usepackage{mathrsfs}

\usepackage{charter}

\usepackage{mathabx, stmaryrd}

\makeatletter
\newcommand{\leqnos}{\tagsleft@true\let\veqno\@@leqno}
\newcommand{\reqnos}{\tagsleft@false\let\veqno\@@eqno}
\reqnos
\makeatother

\numberwithin{equation}{section}

\usepackage{caption}           
\usepackage{url}

\usepackage[a4paper,scale={0.72,0.74},marginratio={1:1},footskip=7mm,headsep=10mm]{geometry}

\usepackage{hyperref}
\hypersetup{					
	breaklinks=true,			
	colorlinks,
    citecolor=black,
    filecolor=black,
    linkcolor=black,
    urlcolor=black
}

\newcommand{\ind}{{\sf 1}}

\newcommand{\bP}{\mathbf{P}}

\newcommand{\bbP}{\mathbb{P}}
\newcommand{\bbE}{\mathbb{E}}
\newcommand{\bbR}{\mathbb{R}}
\newcommand{\bbN}{\mathbb{N}}
\newcommand{\bbZ}{\mathbb{Z}}

\newcommand{\ent}{\mathrm{Ent}}

\newcommand{\sD}{\mathscr{D}}
\newcommand{\sM}{\mathscr{M}}

\newcommand{\cP}{{\ensuremath{\mathcal P}} }
\newcommand{\cE}{{\ensuremath{\mathcal E}} }

\newcommand{\cL}{{\ensuremath{\mathcal L}} }

\newcommand{\cT}{{\ensuremath{\mathcal T}} }

\newcommand{\cG}{{\ensuremath{\mathcal G}} }
\newcommand{\cM}{{\ensuremath{\mathcal M}} }

\definecolor{nicolor}{RGB}{0,102,0}

\newcommand{\dd}{\textrm{d}}

\renewcommand{\epsilon}{\varepsilon}
\renewcommand{\phi}{\varphi}
\renewcommand{\tilde}{\widetilde}

\DeclareMathOperator*{\argmax}{arg\,max}

\newtheorem{theorem}{Theorem}[section]

\newtheorem{proposition}[theorem]{Proposition}
\newtheorem{corollary}[theorem]{Corollary}
\newtheorem{lemma}[theorem]{Lemma}

\theoremstyle{definition}

\newtheorem{remark}[theorem]{Remark}

\newcommand{\ga}{\alpha}
\newcommand{\gb}{\beta}
\newcommand{\gd}{\delta}
\newcommand{\gep}{\varepsilon}       

\newcommand{\go}{\omega}

\newcommand{\gU}{\Upsilon}

\renewcommand{\tilde}{\widetilde}
\renewcommand{\hat}{\widehat}

\newcommand{\maxtwo}[2]{\max_{\substack{#1 \\ #2}}} 
\newcommand{\suptwo}[2]{\sup_{\substack{#1 \\ #2}}}
\newcommand{\inftwo}[2]{\inf_{\substack{#1 \\ #2}}}

\title{Entropy-controlled Last-Passage Percolation}
 
\author[Q. Berger]{Quentin Berger}
\address{Sorbonne Universit\'e, LPSM,
Campus Pierre et Marie Curie, case 188,
4 place Jussieu, 75252 Paris Cedex 5, France}
\email{quentin.berger@sorbonne-universite.fr}

\author[N. Torri]{Niccol\`o Torri}
\address{Sorbonne Universit\'e, LPSM,
Campus Pierre et Marie Curie, case 188,
4 place Jussieu, 75252 Paris Cedex 5, France}
\email{niccolo.torri@sorbonne-universite.fr}

\date{}

\thanks{The authors acknowledge the support of PEPS grant from CNRS.
N. Torri was supported by a public grant overseen by the French National Research Agency (ANR) as part
of the ``Investissements d'Avenir'' program (ANR-11-LABX-0020-01 and ANR-10-LABX-0098).}

\begin{document}

\maketitle

\begin{abstract}
\ In the present article we consider a natural generalization of Hammersley's Last Passage Percolation (LPP) called \emph{Entropy-controlled Last Passage Percolation} (E-LPP), where points can be collected by paths with a global (entropy) constraint which takes in account the whole structure of the path, instead of a local ($1$-Lipschitz) constraint as in Hammersley's LPP.
The E-LPP turns out to be a key ingredient in the context of the directed polymer model when the environment is heavy-tailed, 
which we consider in \cite{cf:BT_HT}. We prove several estimates on the E-LPP in continuous and in discrete settings, which are of interest on their own. We give applications in the context of polymers in heavy-tail environment which are essentials tools in \cite{cf:BT_HT}: we show that the limiting variational problem conjectured in \cite[Conjecture 1.7]{cf:DZ} is finite, and we prove that the discrete variational problem converges to the continuous one, generalizing techniques used in~\cite{AL11,HM07}.
\\[0.1cm]
\end{abstract}

\section{Introduction: Hammersley' LPP and beyond}

\label{sec:ELPP}

%

\label{sec:HammLPP}

Let us recall the original Hammersley's Last Passage Percolation (LPP) problem of the maximal number of points that can be collected by up/right paths, also known as Ulam's problem \cite{Ul61} of the maximal increasing sequence. 

Let $m\in \mathbb N$, and $(Z_i)_{1\le i\le m}$ be $m$ points independently drawn uniformly on the square $[0,1]^2$. We denote the coordinates of these points $Z_i:=(x_i,y_i)$ for $1\le i\le m$.
A sequence $(z_{i_\ell})_{1\leq \ell\leq k}$ is said to be \emph{increasing} if
$x_{i_\ell} >x_{i_{\ell-1}}$ and $y_{i_{\ell}}> y_{i_{\ell-1}}$ for any $1\leq \ell\leq k$ (by convention $i_0=0$ and $z_0=(0,0)$).
The question is  to find the length of the longest increasing sequence among the $m$ points, which is equivalent to finding the length of the longest increasing subsequence of a random (uniform) permutation of length $m$: we let
\begin{equation}\label{HLPP}
\cL_m  = \max \big\{ k\ \colon \exists\, (i_1,\ldots,i_k) \ s.t.\ (Z_{i_\ell})_{1\leq \ell\leq k} \text{ is increasing}\big\}.
\end{equation}

Hammersley  \cite {Ha72} first proved that $m^{-1/2} \cL_m$ converges a.s. and in $L^1$ to some constant, that was believed to be $2$. Then the constant has been proven to be  indeed $2$, see \cite{LS77,VK77}, and estimates related to $\cL_m$ were improved by a series of papers, culminating with a seminal paper by Baik, Deift and Johansson \cite{BDJ99}, showing that $ m^{-1/6} ( \cL_m - 2\sqrt{m}) $
converges in distribution to the Tracy-Widom distribution.

 The main goal of the present article is to define the \emph{Entropy-controlled Last Passage Percolation} (E-LPP),  a natural extension of Hammersley's LPP \eqref{HLPP}. We introduce the concept of \emph{global} (entropy) path constraint, which depends on the structure of the whole path, and is related to the moderate deviation rate function of the simple symmetric random walk.

The E-LPP turns out to be crucial in the analysis of the directed polymer model in a heavy-tailed environment in $(1+1)$-dimension. We refer to \cite{C17, CSY04, FdH07} for the definition of the directed polymer model and a general overview on the main questions. 
Let us stress that among these, a fundamental question is to capture the transversal fluctuations exponent $\xi$ of the polymer. This problem as attracted much attention in recent years,
in particular because the model is in the KPZ universality class: in particular, it is
conjectured that at any fixed inverse temperature $\beta$, the transversal fluctuation exponent is $\xi=2/3$. 
Alberts, Khanin and Quastel \cite{AKQ14a} recently introduced the concept of \emph{intermediate disorder regime} in which $\beta$ scales with $n$, the size of the system. 
In the setting of a heavy-tailed environment, this was considered first by Auffinger-Louidor \cite{AL11}, who showed that rescaling suitably $\beta$, the model has transversal fluctuations of order one, that is $\xi=1$. 
Dey and Zygouras \cite{cf:DZ} then proved that with a different (stronger) rescaling of $\beta$, the model has Brownian fluctuations, that is $\xi=1/2$.
Moreover Dey and Zygouras proposed a phase-diagram picture that connects the exponent of the transversal fluctuation of the polymer $\xi$ with the tail exponent $\alpha$ of the heavy-tailed distribution of the environment and the decay rate of $\gb$. 
In \cite{cf:BT_HT} we start to complete this picture by giving a complete description in the case of $\alpha\in (0,2)$: one of the main results  is a proof of Conjecture 1.7 of \cite{cf:DZ}, describing explicitly the limit, cf.~Theorem~\ref{thm:TbhatTb}.
One crucial tool needed in \cite{cf:BT_HT} is the E-LPP defined below (in the discrete and continuous case), which allows to go beyond the Lipschitz setting of \cite{AL11, HM07}, and treat intermediate transversal fluctuations $1/2<\xi <1$.

Let us highlight that in the related paper \cite{cf:BT_LPP} we investigate further generalizations of Hammersley's LPP problem 
which can bring about new tool and perspectives on this research topic.

 \subsection{Organization of the article}

We state all our results in Section \ref{sec:main}: in Section \ref{sec:defELPP} we give the precise definition of E-LPP and we state our results for the E-LPP in continuous and in discrete settings;
in Section \ref{app1} we consider the problem of E-LPP with heavy-tail weights that appears in \cite{cf:BT_HT}, and we show that the continuous limit in Theorem 2.4 of \cite{cf:BT_HT}
is well defined, completing the proof of \cite[Conjecture 1.7]{cf:DZ}; in Section \ref{sec:OrderedStat} we state the convergence of the discrete energy-entropy variational problem to its continuous counterpart. This result is crucial to prove the convergence in Theorems~2.2--2.7 of \cite{cf:BT_HT}. 
The proofs of the all results are presented in Sections \ref{app:LPP} to \ref{sec:proofdisc}.

\section{Main Results}\label{sec:main}

Operating a rotation by $45^\circ$ clockwise, we may map Hammersley's LPP problem (cf. Section \ref{sec:HammLPP}) to that of the maximal number of points that can be collected by $1$-Lipschitz paths $s:[0,1]\to\bbR$.
We now introduce a new (natural) model where the Lipschitz constraint is replaced by a path entropy constraint.

\subsection{Entropy-controlled LPP}
\label{sec:defELPP}

For $t>0$, and a finite set $\Delta = \big\{ (t_i,x_i) ; 1\le i\le j \big\} \subset [0,t] \times \bbR$ with $|\Delta|=j\in \mathbb N$ and with $0\le t_1\le  t_2\le \cdots \le t_j \le t$, we can define the entropy of $\Delta$ as
\begin{equation}
\label{def:entdelta}
\ent(\Delta) := \frac12 \sum_{i=1}^j \frac{(x_i-x_{i-1})^2}{t_i-t_{i-1}} \, ,
\end{equation}
where we used the convention that $(t_0,x_0)=(0,0)$. If there exists some $1\le i\le j$ such that $t_i=t_{i-1}$ then we set $\ent(\Delta)=+\infty$. This corresponds to the definition \eqref{def:ContinuumEntropy} of the entropy of a continuous path $s:[0,t]\to \mathbb R$, applied to the linear interpolation of the points of $\Delta$:  to any set $\Delta$ we can therefore  canonically associated a (continuous) path  with the same entropy.
The set $\Delta$ is seen as a set of points a path has to go through. For $S=(S_i)_{i\geq 0}$ a simple symmetric random walk on $\mathbb{Z}$, and if  $\Delta \subset \bbN \times \bbZ$, we have that $\bP(\Delta \subset S) \le e^{-\ent(\Delta)}$  ($\Delta\subset S$ means that $S_{t_i}=x_i$ for all $i\le |\Delta|$)---we used that for the simple random walk $\bP(S_i=x) \le e^{- x^2/2i}$ by a standard Chernoff bound argument.

Then, for any fixed $B>0$, we will consider the maximal number of points that can be collected by paths with entropy smaller than $B$, among a random set $\gU_m$ of $m$ points, whose law is denoted $\bbP$.
We now consider two types of problems, depending on how this set $\gU_m$ is constructed:
\begin{itemize}
\item[(i)] \emph{continuous} setting: for $t,x>0$, we consider a domain $\boldsymbol{\Lambda}_{t,x}:=[0,t] \times [-x,x]$, and $\boldsymbol \gU_m = \boldsymbol \gU_m(t,x)= \{Y_1,\ldots, Y_m\}$ where $(Y_i)_{1\le i\le m}$ is a collection of independent r.v.\ chosen uniformly in  $\boldsymbol \Lambda_{t,x}$;

\item[(ii)] \emph{discrete} setting: for $n,h\in \mathbb{N}$, we consider a domain $\Lambda_{n,h} := \llbracket 0,n \rrbracket \times \llbracket -h,h \rrbracket
$,
and $\gU_m  =\gU_m (n,h)= \{Y_1,\ldots, Y_m\}$ is a set of $m$ distinct points taken randomly in~$\Lambda_{n,h}$.
\end{itemize}

We are then able to define the Entropy-controlled LPP by
\begin{equation}\label{def:LcontLdis}
\mathcal{L}_{m}^{(B)}(t,x) = \maxtwo{\Delta \subset \boldsymbol\gU_m(t,x)}{\ent(\Delta)\le B}  \big| \Delta \big| \, , \qquad L_{m}^{(B)}(n,h) = \maxtwo{\Delta \subset \gU_m(n,h)}{\ent(\Delta)\le B}  \big| \Delta \big| \, ,
\end{equation}
the maximal number of points than can be included in a set $\Delta$ that has entropy smaller than~$B$. In other words, it is the maximal number of points in  $\boldsymbol \gU_m$ or $\gU_m$ that can be collected by a path of entropy smaller than $B$. 
note that we use the different font to be able to differentiate the setting: $\mathcal{L}, \boldsymbol \Lambda, \boldsymbol \gU$ for the continuous case and $L, \Lambda, \gU$ for the discrete one.

We show the following result---the lower bound is not needed for our applications, but can be found in~\cite{cf:BT_LPP}.
\begin{theorem}
\label{thm:generalLPP}
There are constants $C_0,c_0, c'_0>0$ such that: for any $t,x,B >0$, $n,h\ge 1$
\smallskip

\emph{(i) continuous setting:} for all $m\ge 1$ and all $k\le m$
\begin{align}
 \bbP \Big( \mathcal{L}_{m}^{(B)} (t,x) \ge k \Big) & \leq   \Big(  \frac{C_0 (Bt/x^2)^{1/2} m  }{k^2}\Big)^k \,  .
 \label{thm31UB}
\end{align}

\emph{(ii) discrete setting:} for all $1\le m\le nh$ and all $k\le m$
\begin{align}
  \bbP \Big(  {L}_{m}^{(B)} (n,h) \ge k \Big) &\leq  \Big(  \frac{C_0 (Bn/h^2)^{1/2} m  }{k^2}\Big)^k \,  .
\end{align}
\end{theorem}
The proof of Theorem~\ref{thm:generalLPP} is not difficult but a bit technical, and we give it in Section~\ref{app:LPP}.
This result shows in particular that $\mathcal L_m^{(B)}(t,x)$ is of order $\big( ( Bt/x^2)^{1/4}  \sqrt{m} \big)\wedge m$, resp.\ $L_m^{(B)}(n,h)$ is of order $ \big( ( Bn/h^2)^{1/4}  \sqrt{m} \big)\wedge m$, as stressed by the following corollary.
We stress that keeping track of the dependence in $B$ is essential for the applications we have in mind.
\begin{corollary}
\label{cor:LPP}
For any $b>0$, there is a constant $c_b>0$ such that, for any $m\ge 1$, and any positive $B$, and any $t,x$, resp.\ $n,h$,
 \[ \begin{split} \bbE\bigg[ \bigg(  \frac{ \mathcal L_m^{(B)}(t,x)  }{\big( ( Bt/x^2)^{1/4}  \sqrt{m} \big)\wedge m } \bigg)^b \bigg] \le c_b \, ; \ \bbE\bigg[ \bigg(  \frac{  L_m^{(B)}(n,h)  }{\big( ( Bn/h^2)^{1/4}  \sqrt{m} \big)\wedge m } \bigg)^b \bigg] \le c_b \, .
 \end{split}\]
\end{corollary}

%

\begin{remark}\rm
On may view Theorem \ref{thm:generalLPP} as a generalization of \cite[Proposition 3.3]{HM07}. 
More precisely, we recover \cite[Proposition 3.3]{HM07} by considering $\Lambda_{n,n}=\llbracket n,n \rrbracket^2$ and replacing the entropy condition $\ent(\Delta)\le B$ by a \emph{Lipschitz} condition, that is considering only the sets $\Delta$ whose points can be interpolated using a Lipschitz path. Let us denote $L_m^{(\text{Lip})}(n)$ the LPP obtained. Now observe that if $\Delta$ satisfies the Lipschitz condition we have that $\ent(\Delta)\le n/2$ (recall the definition \eqref{def:entdelta}): as a consequence it holds that $L_m^{(n/2)}(n,n)\ge L_m^{(\text{Lip})}(n)$.
We also stress that our definition of E-LPP opens the way to many extensions: in particular as soon as one is able to properly define the entropy of a path (\textit{i.e.}\ of a set $\Delta$), one could extend the results to the case of  paths with unbounded jumps or even non-directed paths: this is the object of \cite{cf:BT_LPP}, where a general notion of path-constrained LPP is developed and studied.
\end{remark}

Let us stress here that one might want to reverse the point of view, and estimate the minimal entropy needed for a path to visit at least $k$ points. This turns out to be essential in Section 4 of \cite{cf:BT_HT}. One realizes that
\[   \inftwo{\Delta \subset \gU_m }{ |\Delta| \ge k} \ent (\Delta) \le B \qquad \Longleftrightarrow \qquad \suptwo{\Delta\subset \gU_m}{ \ent(\Delta) \le B} |\Delta| \ge k \, . \]
Hence, an easy consequence of Theorem \ref{thm:generalLPP} is that for any $k\le n$ (we state it only in the discrete setting)
\begin{equation}
\label{eq:infentropy}
\bbP\Big( \inf_{\Delta \subset \gU_m , |\Delta| \ge k} \ent (\Delta) \le B \Big) \le \Big(  \frac{C_0 (Bn/h^2)^{1/2} m  }{k^2}\Big)^k \, .
\end{equation}
It therefore says that, with high probability, a path that collects $k$ points in $\gU_m \subset \Lambda_{n,h}$ has an entropy larger than  a constant times $k^4/m^2 \times h^2/n $.

\subsection{Application I: continuous E-LPP with heavy-tail weights}
\label{app1}

In \cite{cf:BT_HT} we  prove the convergence of the directed polymer model in heavy-tail environment (suitably rescaled) to a continuous energy-entropy variational problem $\cT_\beta$, defined below in \eqref{def:T} (or in Section~2.2 of \cite{cf:BT_HT}).
A first application of our E-LPP is to show that this variational problem is well-defined when the tail decay exponent $\alpha$ is in $(1/2,2)$: this is Theorem~\ref{thm:TbhatTb}, which proves the first part of \cite[Conjecture~1.7]{cf:DZ}. The second part of this conjecture, \textit{i.e.}\ that $\cT_{\gb}$ is indeed the scaling limit of the directed polymer in heavy-tail environment, is proved in \cite[Theorem~2.4]{cf:BT_HT}. 

\smallskip


Let us recall some notations from Section~2.2 in \cite{cf:BT_HT}.
The set of allowed paths (scaling limits of random walk trajectories) is
\begin{equation}
\label{def:D}
\sD := \big\{ s: [0,1]\to \bbR\ ;\ s \text{ continuous and a.e.\ differentiable} \big\}\, ,
\end{equation}
and the (continuum) \emph{entropy} of a path $s\in \sD$ is defined by
\begin{equation}
\label{def:ContinuumEntropy}
\ent(s) = \frac12 \int_0^1 \big( s'(t) \big)^2 dt \, .
\end{equation}
This last definition derives from the rate function of the moderate deviation of the 
simple random walk (see \cite{S67} or \cite[Eq. (2.14)]{cf:BT_HT}).

We let $\cP:=\{(w_i,t_i,x_i)\}_{ i\geq 1}$ be a Poisson Point Process on  $[0,\infty)\times[0,1]\times\mathbb R $
of intensity $\mu(\dd w \dd t \dd x)=\frac{\alpha}{2} w^{-\alpha-1}\ind_{\{w>0\}}\dd w \dd t \dd x$, where $\alpha\in (0,2)$.
For a quenched realization of $\cP$, the \emph{energy} of a continuous path $s\in\sD$ is then defined by
\begin{equation}
\label{def:discrCont}
\pi(s) =\pi_{\cP}(s):=\sum_{(w,t,x)\in \cP} \, w \,\ind_{\{(t,x)\in s\}},
\end{equation}
where $(t,x)\in s$ means that $(t,x)$ is visited by the path $s$, that is $s_t=x$.

Using \eqref{def:ContinuumEntropy} and \eqref{def:discrCont} we define the \emph{energy--entropy competition variational problem}: 
for any $\gb\geq 0$ we let
\begin{equation}
\label{def:T}
\mathcal T_{\gb}  := \sup_{s\in \sD, \ent(s) <+\infty} \Big\{ \gb \pi(s) -\ent(s) \Big\} .
\end{equation}
The next result shows that it is well defined, and gives some of its properties.
\begin{theorem}\label{thm:TbhatTb}
For $\ga\in (1/2,2)$ we have the scaling relation
\begin{equation}
\label{scalingrelation}
 \cT_{\gb} \stackrel{(\dd)}{=} \gb^{\tfrac{2\ga}{2\ga -1}}\,  \cT_{1} ,
\end{equation}
and $\cT_{\gb}\in (0,+\infty)$ for all $\gb>0$ a.s. Moreover, $\bbE\big[ (\cT_{\gb})^{\upsilon} \big] < \infty$ for any  $\upsilon< \ga-1/2$.
We also have that a.s.\ the map $\gb\mapsto \cT_{\gb}$ is continuous, and that the supremum in~\eqref{def:T} is attained by some unique continuous path $s_{\gb}^*$ with $\ent(s_{\gb}^*)<\infty $.

On the other hand, for  $\ga\in(0,1/2]$ we have $\cT_{\gb} =+\infty$ for all $\gb>0$ a.s. 
\end{theorem}

\begin{remark}\rm
\label{rem:uniqueness}
As we discuss in Section 2.5 of \cite{cf:BT_HT}, the fact that the maximizer of $\cT_\beta$ is unique could be used to show the concentration of the paths around $s_{\gb}^{*}$ under the polymer measure $\bP_{n,\beta_n}^\omega$, in analogy with the result obtained by Auffinger and Louidor in Theorem 2.1 of \cite{AL11}.
%
\end{remark}

\subsection{Application II: discrete E-LPP with heavy-tail weigths}
\label{sec:OrderedStat}

In this section we discuss the convergence of a discrete energy-entropy variational problem $T_{n,h}^{\gb_{n,h}}$ defined below \eqref{def:discreteELPP}, to its continuous counterpart $\cT_\beta$ \eqref{def:T}. This is a crucial result that we need in \cite{cf:BT_HT} to prove Theorems 2.4--2.7. 

We introduce the discrete field $\{ \omega_{i,x} ; {(i,x)\in \mathbb N\times\mathbb Z}\}$, which are i.i.d.\ non-negative random variables of law $\mathbb P$: there is some slowly varying function $L(\cdot)$ and some $\ga>0$ such that
\begin{equation}\label{eq:disorder}
\mathbb P\big(\omega >x\big)=L(x) x^{-\alpha}\, .
\end{equation}
This random field is the discrete counterpart of the Poisson Point Process~$\cP$ introduced in Section \ref{app1}. We refer to Section \ref{contLimitHT} for further details.

Let us consider $F(x) = \bbP(\go\le x)$ be the disorder distribution, cf. \eqref{eq:disorder}, and define the function $m(x)$ by
\begin{equation}
\label{def:m}
m(x) := F^{-1} \big(1-\tfrac1x \big), \qquad \text{so  }\ \bbP \big( \go >m (x) \big) = 1/x  .
\end{equation}
The second identity characterizes $m(x)$ up to asymptotic equivalence: we have that $m(\cdot)$ is a regularly varying function with exponent $1/\ga$.

\smallskip

%

For any given box $\Lambda_{n,h} = \llbracket 1,n \rrbracket \times \llbracket -h,h \rrbracket$ we can rewrite the discrete field in this region $(\omega_{i,x})_{(i,x)\in \Lambda_{n,h}}$ using the \emph{ordered statistic}: we let 
$M_r^{(n,h)}$ be the $r$-th largest value of $(\omega_{i,x})_{(i,x)\in \Lambda_{n,h}}$ and $Y_r^{(n,h)}\in \Lambda_{n,h}$ its position---note that $(Y_r^{(n,h)})_{r=1}^{|\Lambda_{n,h}|}$ is simply a random permutation of the points of $\Lambda_{n,h}$. 
In such a way 
\begin{equation}\label{eq:ordstat}
(\omega_{i,j})_{(i,j)\in \Lambda_n} = (M_r^{(n,h)},Y_r^{(n,h)})_{r=1}^{|\Lambda_{n,h}|} \, .
\end{equation}
In the following we refer to $(M_r^{(n,h)})_{r=1}^{|\Lambda_{n,h}|}$ as the \emph{weight} sequence. 
We now define the energy collected by a set $\Delta \subset \Lambda_{n,h}$ 
and its contribution by the first $\ell$ weights (with $1\le \ell \le |\Lambda_{n,h}|$) as follows
\begin{equation}
\label{def:Omega}
\Omega_{n,h} (\Delta)  := \sum_{r=1}^{|\Lambda_{n,h}|} M_r^{(n,h)} \ind_{\{ Y_r^{(n,h)} \in \Delta\}}\, ; \  \ \Omega_{n,h}^{(\ell)} (\Delta)  := \sum_{r=1}^{\ell} M_r^{(n,h)} \ind_{\{ Y_r^{(n,h)} \in \Delta\}}\, .
\end{equation}
We also set $ \Omega_{n,h}^{(>\ell)} (\Delta) :=  \Omega_{n,h} (\Delta)-  \Omega_{n,h}^{(\ell)}(\Delta)$.

In such a way we can define the (discrete) variational problem
\begin{equation}
\label{def:discreteELPP}
T_{n,h}^{\gb_{n,h}} := \max_{ \Delta \subset \Lambda_{n,h}} \big\{  \gb_{n,h} \Omega_{n,h} (\Delta) - \ent(\Delta)  \big\} \, ,
\end{equation}
with $\gb_{n,h}$ some  function of $n,h$ (soon to be specified), and $\ent(\Delta)$ as defined in~\eqref{def:entdelta}.
We also define analogues of \eqref{def:discreteELPP} with a restriction to the $\ell$ largest weights, or beyond the $\ell$-th  weight
\begin{equation}
\label{def:discrELPPell}
\begin{split}
T_{n,h}^{\gb_{n,h},(\ell)} &:= \max_{ \Delta \subset \Lambda_{n,h}} \big\{  \gb_{n,h} \Omega_{n,h}^{(\ell)} (\Delta) - \ent(\Delta)  \big\} \, ,\\
T_{n,h}^{\gb_{n,h},(>\ell)} &:= \max_{ \Delta \subset \Lambda_{n,h}} \big\{  \gb_{n,h} \Omega_{n,h}^{(>\ell)} (\Delta) - \ent(\Delta)  \big\} \, .
\end{split}
\end{equation}

The following proposition is crucial for the proof of Theorem~\ref{prop:ConvVP} below, and is also a central tool  in \cite[Section~4]{cf:BT_HT}.
\begin{proposition}
\label{thm:discreteELPP}
The following hold true:

\textbullet\ For any $a<\ga$, there is a constant $c_a>0$ such that for any $1\le \ell\le nh$, for any $b > 1$
\begin{equation}
\bbP \Big(   T_{n,h}^{\gb_{n,h}, (\ell)} \ge b \times  {\big( \gb_{n,h} m(nh) \big)^{4/3}}{ \Big(\frac{n}{h^2}\Big)^{1/3}}   \Big) \le  c_a\,  b^{-3a/4}\, \label{smallthanell}.
\end{equation}

\textbullet\ We also have that there is a constant $c>0$ such that for any $b> 1$
\begin{equation}
\bbP \Big(   T_{n,h}^{\gb_{n,h}, (>\ell)} \ge b \times { \big( \gb_{n,h} m(nh/\ell) \big)^{4/3}}{\Big(\frac{\ell^2n}{h^2} \Big)^{1/3}}      \Big) \le  c b^{-\ga \ell/4} + e^{-c b^{1/4}}\, .\label{bigthanell}
\end{equation}

\end{proposition}

The proof is is postponed to Section~\ref{proofthmELPPd}. Observe that we need here to keep track of the dependence on $n,h$: to that end, estimates obtained in Section~\ref{sec:ELPP} will be crucial.
Note already that if $\frac{n}{h^2}\gb_{n,h} m(nh) \to \gb \in (0,\infty)$, as $n,h\to\infty$, it gives that $T_{n,h}^{\gb_{n,h}, (\ell)}$ is of order $\gb^4 h^2/n$.

\smallskip

In the next result we prove the convergence in distribution for  \eqref{def:discreteELPP}, which generalizes the convergence of  related variational problems considered in~\cite{AL11,HM07}.

\begin{theorem}\label{prop:ConvVP}
Suppose that $\frac{n}{h^2}\gb_{n,h} m(nh) \to \nu \in[0,\infty)$ as $n,h\to\infty$. For every $\ga\in (1/2,2)$ and for any $q>0$ we have the following convergence in distribution
\begin{equation}
\label{def:TA}
\frac{n}{h^2}\,  T_{n,qh}^{\beta_{n,h}} \xrightarrow[n\to\infty]{(\dd)} \cT_{\nu,q} :=\sup_{s\in \sM_q}\big\{\nu \pi(s)-\ent(s) \big\} \,,
\end{equation}
with $\sM_q:=\{s\in \sD , \ent(s)<\infty, \max_{t\in [0,1]} |s(t)|\le q\}.$
We also have
\begin{equation}
\label{conv:largeweights}
\frac{n}{h^2}\,  T_{n,qh}^{\beta_{n,h}, (\ell)} \xrightarrow[n\to\infty]{(\dd)} \cT_{\nu,q}^{(\ell)} :=\sup_{s\in \sM_q}\big\{\nu \pi^{(\ell)}(s)-\ent(s) \big\} \,,
\end{equation}
where $\pi^{(\ell)}:= \sum_{r=1}^{\ell} M_r \ind_{\{Y_r\in s\}}$ with $\{(M_r,Y_r)\}_{r\ge 1}$ the ordered statistics of $\cP$ restricted to $[0,1]\times [-q,q]$, see Section \ref{contLimitHT}.
Finally, we have
\begin{equation}\label{Tconverge}
\cT_{\nu,q}^{(\ell)} \xrightarrow[\ell\to\infty]{a.s.} \cT_{\nu,q}, \quad 
\qquad \text{and}\qquad \cT_{\nu,q} \xrightarrow[q\to\infty]{a.s.} \cT_{\nu}.
\end{equation}
\end{theorem}

\section{Proof of Theorem~\ref{thm:generalLPP} and Corollary~\ref{cor:LPP}}
\label{app:LPP}

\subsection{Proof of Theorem~\ref{thm:generalLPP}}

We start with the proof in the continuous setting. The discrete setting follows the same lines and details will be skipped.

\subsubsection*{Continuous setting}
Let us consider $\cE_{k}^{(t,B)}$ the set of $k$-uples in $[0,t]\times \bbR$ (\textit{i.e.}\ up to time $t$) that have entropy smaller than $B$:
\[
\cE_{k}^{(t,B)} = \Bigg\{ \big( (t_\ell, x_{\ell}) \big)_{1\leq \ell\leq k} \subset [0,t]\times \bbR\, ; 
   \begin{aligned} &\ 0<  t_1<\cdots <t_k < t \ ;\\ &\ \ent\big(  (t_\ell, x_{\ell}  )_{1\leq \ell\leq k}\big) \le B
   \end{aligned}
\Bigg\}\, .
\]

We can compute exactly the volume of $\cE_k^{(t,B)}$.
\begin{lemma}
\label{lem:volume}
We have, for any $t>0$ and $B> 0$
\[\mathrm{Vol} \big( \cE_k^{(t,B)} \big) =   C_k  \times B^{k/2} t^{3 k /2} ,\quad \text{with } C_k =  
\frac{ \pi^k / \sqrt{2} }{ \Gamma \big( k/2 +1 \big) \Gamma \big( 3k/2+1 \big)} \, .\]
In particular, it gives that there exists some constant $C$ such that 
\[\mathrm{Vol}\big( \cE_k^{(t,B)} \big)  \leq \Big(  \frac{C B^{1/2} t^{3/2}  }{k^2}\Big)^k \, .\]
\end{lemma}

\begin{proof}
The key to the computation is the induction formula below, based on the decomposition over the left-most point in $\cE_k^{(t,B)}$ at position $(u,y)$ (by symmetry we can assume $y\geq 0$): it leaves $k-1$ points with remaining time $t-u$ and entropy smaller than $B-  \tfrac{y^2}{ 2 u}$,
\begin{equation}
\label{eq:induction}
\mathrm{Vol}\big( \cE_k^{(t,B)} \big) = 2 \int_{u=0}^t \int_{y=0}^{\sqrt{2B u}} \mathrm{Vol} \big( \cE_{k-1}^{(t-u , B - y^2/ 2u )} \big) dy du.
\end{equation}
The induction is only calculations.
For $k=1$ we have
\[\mathrm{Vol}\big( \cE_1^{(t,B)} \big)  = 2 \int_{u=0}^t \int_{y=0}^{\sqrt{2B u}}  du dy = 2\sqrt{2B}  \int_0^t u^{1/2} du = \frac{4\sqrt{2}}{3} B^{1/2} t^{3/2} \, ,\]
so that we indeed have that 
$C_1 = \pi  ( \sqrt{2}\Gamma(3/2)\Gamma(5/2) )^{-1}$.

For $k\geq 2$, by induction, we have
\begin{align*}
\mathrm{Vol}(\cE_k^{(t,B)}) = 2 C_{k-1} \int_{u=0}^t \int_{y=0}^{\sqrt{2B u} }   (t-u)^{3(k-1)/2} \big( B-\tfrac{y^2}{2u} \big)^{(k-1)/2}  dy du.
\end{align*}
Then, by a change of variable $w= y^2 /(2B u)$, we get that
\begin{align*}
\int_{y=0}^{\sqrt{2B u}} \big( B-\tfrac{y^2}{2u} \big)^{(k-1)/2}  dy
&= B^{(k-1)/2} \int_0^1 (1-w)^{(k-1)/2} \, \sqrt{\frac{Bu}{2}}  w^{-1/2}  dw  \\
& = \frac{1}{\sqrt{2}} B^{k/2} \, u^{1/2} \, \frac{\Gamma\big( (k-1)/2+1 \big) \Gamma(1/2)}{\Gamma(k/2+1)}\, .
\end{align*}
Moreover, we also have
\begin{align*}
 \int_{u=0}^t  u^{1/2} (t-u)^{3(k-1)/2} dx &= t^{3(k-1)/2+ 1/2+1} \int_0^1 v^{1/2} (1-v)^{3(k-1)/2} dv \\
 &=  t^{3 k/2}  \frac{\Gamma(3/2) \Gamma( 3(k-1)/2 +1 )}{ \Gamma( 3 k/2 +1) }\, .
\end{align*}
Hence, the constant $C_k$ verifies
\begin{align*}
C_{k}& =  2 C_{k-1} \times \sqrt{\pi  } \frac{\Gamma\big( (k-1)/2+1 \big)}{ \Gamma(k/2+1)} \times  {\frac{\sqrt \pi }{2} } \frac{\Gamma(3(k-1)/2 +1)}{ \Gamma(3k/2 +1)}\,,
\end{align*}
which completes the induction, in view of the formula for $C_{k-1}$.

For the inequality in the second part of the lemma, we simply use Stirling's formula to get that there is a constant $c>0$ such that
\[ \Gamma \big( k/2 +1 \big) \geq \big( c k \big)^{k/2} \quad \text{ and } \quad \Gamma \big( 3k/2 +1  \big) \ge \big( c k \big)^{3k/2} \, .\]
\end{proof}

Let us denote $\mathcal N_k$ the number of sets $\Delta \subset \gU_m (t,x)$ with $|\Delta|=k$, that have entropy at most~$B$. We write
\[\bbP\big( \mathcal {L}_m^{(B)}(t,x)\geq k \big) = \bbP(\mathcal N_k \geq 1) \leq \bbE[\mathcal N_k] \, .\]
Since all the points are exchangeable, we get 
\[\bbE[\mathcal N_k] = \binom{m}{k} \bbP\Big( \exists \ \sigma\in \mathfrak{S}_k\ s.t. \  (Z_{\sigma(1)},\ldots,  Z_{\sigma(k)}) \in \cE_{k}^{(t,B)} \Big),\]
where $ Z_1 = (t_1, x_1), \ldots , Z_k = (t_k,x_k)$ are independent uniform r.v. on the domain $\boldsymbol \Lambda_{t,x}$ (with volume $2tx$).
We then have that
\[\bbP \Big(  \exists \ \sigma\in \mathfrak{S}_k\ s.t. \  (Z_{\sigma(1)},\ldots,  Z_{\sigma(k)}) \in \cE_{k}^{(t,B)}\Big) = k! \, \frac{\mathrm{Vol}(\cE_{k}^{(t,B)})}{ (2tx)^k} \, . \]
We therefore obtain, using that $\binom{m}{k} \leq m^k /k!$,  together with Lemma \ref{lem:volume}
\begin{equation}
\bbP\big( \mathcal {L}_m^{(B)}(t,x)\geq k \big)  
\leq \Big(  \frac{C B^{1/2} t^{1/2} m  }{2 x k^2}\Big)^k 
\end{equation}
This gives the upper bound of Theorem \ref{thm:generalLPP}-(i). \qed

\subsubsection*{Discrete setting: upper bound}

The proof follows the same idea as above: we skip most of the details.
Define $E_{k}^{(n,B)}$ the set of $k$-uples in $\llbracket 1,n \rrbracket\times \bbZ$  that have entropy smaller than $B$:
\[
E_{k}^{(n,B)} := \Bigg\{ \big( (t_\ell, x_{\ell}) \big)_{1\leq \ell\leq k} \subset \llbracket 1,n \rrbracket \times \bbZ \ ;
\begin{aligned} & \ 0<  t_1<\cdots <t_k \le n \, ;
\\& \ \ent\big(  (t_\ell, x_{\ell}  )_{1\leq \ell\leq k}\big) 
  \le B
  \end{aligned} \Bigg\} \, .\]

We can estimate the cardinality of $E_{k}^{(n,B)}$---however not in an exact manner as in the continuous case.
\begin{lemma}
\label{lem:voldiscrete}
For any $n\in \mathbb N$ it holds true that
\[\mathrm{Vol} \big( E_{k}^{(n,B)} \big) \le   
2^k \, C_k  \times B^{k/2} n^{3 k /2} ,\quad \text{with } C_k =  
\frac{ \pi^k / \sqrt{2} }{ \Gamma \big( k/2 +1 \big) \Gamma \big( 3k/2+1 \big)} \, .\]
\end{lemma}

\begin{proof}
The analogous of \eqref{eq:induction} is here
\begin{equation}
\mathrm{Vol}\big( E_k^{(n,B)} \big) = 2 \sum_{i=1}^n \sum_{y=0}^{\sqrt{2B i}} \mathrm{Vol} \big(E_{k-1}^{(n-i,B-x^2/2i)} \big).
\end{equation}
The induction is again straightforward calculations: we can use the computations made in the continuous setting, together with the comparison between finite sums and Riemann integrals, \textit{i.e.} 
\begin{equation}\label{eq:sumRiemann}
\begin{split}
&\sum_{i=0}^n g(i)\le \int_0^{n+1} g(z)\dd z\qquad \text{if $g$ is increasing},\\
&\sum_{i=0}^n g(i)\le g(0)+\int_0^{n} g(z)\dd z\qquad \text{if $g$ is decreasing} \, .
\end{split}
\end{equation}
Details are left to the reader.
\end{proof}

Again, we have
$\bbP \big( L_m^{(B)}(n,h)\geq k \big)  \leq \bbE[N_{k}]$,
where $N_k$ is the number of sets $\Delta \subset \gU_{m} \subset \Lambda_{n,h}$ with $|\Delta|=k$, that have entropy at most $B$.
Then,
\[\bbE[N_{k}] = \binom{m}{k} 
\bbP\Big( \exists \ \sigma\in \mathfrak{S}_k\ s.t. \  (Z_{\sigma(1)}^{(n,h)},\ldots,  Z_{\sigma(k)}^{(n,h)}) \in \cE_{k}^{(n,B)} \Big)
,\]
where $ (Z_1^{(n,h)},\cdots, Z_k^{(n,h)} )$ are a uniform random choice of $k$ distinct points from $\Lambda_{n,h}$ (which contains $n(2h+1)$ points)---the main difference with the continuous setting comes from the fact that the $Z_i$'s are not independent.
We therefore have that, using Lemma~\ref{lem:voldiscrete},
\[\bbE[n_{k}] = \binom{m}{k}  \ 
\frac{ \mathrm{Vol}(E_{k}^{(n,B)})}{\binom{2nh+n}{k}}  \le \frac{m^k}{(2nh)^k} \Big(\frac{C B^{1/2} }{ k^2} \Big)^k \, . \]
We also used that $\binom{m}{k} \le m^k /k!$ and that $\binom{2nh+n}{k} \ge  (2nh +n -k)^k/k!$ with $k\le n$.
This concludes the proof of the upper bound in Theorem~\ref{thm:generalLPP}-(ii).\qed

\subsection{Proof of Corollary~\ref{cor:LPP}}
We prove it in the continuous setting, the discrete one being similar.
From Theorem \ref{thm:generalLPP}, we deduce that for any $u\ge  (e C_0)^{1/2}$, we have
\begin{equation}
\label{tailL}
\bbP\Big(  \mathcal L_m^{(B)}(t,x)   \ge u  \, (Bt/x^2)^{1/4}  \sqrt{m} \Big) \le \exp \Big( - u (Bt/x^2)^{1/4}  \sqrt{m}  \Big) . 
\end{equation}
Applying this inequality with $u = (e C_0)^{1/2}$, and using also the \emph{a priori} bound $\mathcal L_m^{(B)}(n,h)\le m$, we get that for any $b>0$
\begin{align*}
\bbE\bigg[ \bigg( &\frac{\mathcal L_m^{(B)}(t,x)}{\big( (Bt/x^2)^{1/4} m^{1/2}\big) \wedge m} \bigg)^b \bigg] \\
& \le (e C_0)^{b/2}+ \int_{ ( e C_0)^{b/2} }^{+\infty} \bbP\bigg(   \frac{\mathcal L_m^{(B)}(t,x)}{\big( (Bt/x^2)^{1/4} m^{1/2}\big) \wedge m}  > u^{1/b}\Big) d u\\
&\le   (e C_0)^{b/2} + cst.\, 
\end{align*}

%

\section{Proof of Theorem \ref{thm:TbhatTb}}
\label{ProofThmTbhatTb}


Let us recall that $\cP:=\big\{(w_i,t_i,x_i)\colon i\geq 1 \big\}$ is a Poisson Point Process  on $[0,\infty) \times[0,1]\times \mathbb R$
of intensity $\mu(\dd w  \dd t \dd x)=\frac{\alpha}{2} w^{-\alpha-1}\ind_{\{w>0\}} \dd w  \dd t \dd x$, as introduced in Section \ref{app1}. 

\subsection{Ideas of the proof} First we prove that $\cT_{\gb}=+\infty$ when $\ga\le 1/2$. Then,  we prove the scaling relation \eqref{scalingrelation}, and finally we show the finiteness of the $\upsilon$-th moment ($\upsilon <\ga-1/2$). We stress that the core of the proof is based on an application of the continuous E-LPP: roughly, the idea of the proof is to decompose the variational problem \eqref{def:T} according to the value of the entropy:
\begin{equation}
\label{varB}
\cT_{\gb} = \sup_{B\ge 0} \Big\{ \gb \sup_{s\in \sD, \ent(s) =B} \pi(s)    -B \Big\} \, .
\end{equation}
Then, a simple scaling argument gives that
\[\sup_{s:\ent(s) \le B} \pi(s) \stackrel{(d)}{=} B^{\frac{1}{2\ga}}\sup_{s:\ent(s) \le 1} \pi(s).\] 
The E-LPP appears essential to show that the last supremum is finite, see in particular Lemma~\ref{lem:moment} below. Then, at a heuristic level, we get that $\cT_{\gb}$ is finite because in \eqref{varB} we have $B^{\frac{1}{2\ga}}\ll B$ as $B\to\infty$ (remember that $\ga>1/2$).
In the last part of the proof we prove the continuity of $\gb\mapsto \cT_{\gb}$ and of the existence and uniqueness of the maximizer in \eqref{def:T}. 

\subsection{Case $\ga\le 1/2$}
Let us prove here that $\cT_{\gb} =+\infty$ when $\ga\in (0,1/2]$. 
For any $k$ in $\bbZ$, we define the event
\[\cG_k := \big\{\cP\cap  [ \gb^{-1} 2^{2k+1},+\infty) \times [\tfrac12,1] \times [2^{k-1},2^{k}]   \neq \emptyset \big\} \, ,\]
On the event $\cG_k$, we denote $(w_k,t_k,x_k)$ a point of $\cP$ such that $w_k \ge \gb^{-1} 2^{2k+1}$ and $(t_k,x_k)\in [\tfrac12,1] \times [2^{k-1},2^{k}]$: considering the path going straight to $(t_k,x_k)$ we get that 
\[  \cT_{\gb} \ge \gb w_k - \frac{x_k^2}{2t_k}  \ge  2^{2k}, \qquad    \text{ on the event }\cG_k \, .\]
Then, it is just a matter of estimating $\bbP(\cG_k)$. We stress that considering $\cM_k$ the maximal weight in $[\tfrac12,1] \times [2^{k-1},2^{k}]$, we find that $\cM_k$ is of order $(2^k)^{1/\ga}$ (as a maximum of a field of independent heavy-tail random variables, or using the scaling relations below), so that we get that:
if $\ga<1/2$, $\bbP(\cG_k) \to 1$ as $k\to+\infty$; if $\ga=1/2$, there is a constant $c>0$ such that $\bbP(\cG_k)\ge c$ for all $k\in \bbZ$; if $\ga>1/2$, $\bbP(\cG_k) \to 1$ as $k\to -\infty$.
note that  the events $\cG_k$ are independent, so an application of Borel-Cantelli lemma gives that for $\ga\le 1/2$, a.s.\ $\cG_k$ occurs for infinitely many $k \in \bbN$: since  $\cT_{\gb}\ge 2^{2k}$ on $\cG_k$, it implies that $\cT_{\gb} = +\infty$ a.s.\ for $\ga\le 1/2$.

On the other hand, it also proves that when $\ga>1/2$, a.s.\ there exists some $k_0 \le -1$ such that $\cG_{k_0}$ occurs and thus $\cT_{\gb} \ge 2^{2k_0}>0$.

\subsection{Scaling relations}
For any $\alpha\in (0,2)$ and $a>0$ we consider two functions $\phi(w,t,x) := (w,t,a x)$ and $\psi(w,t,x) := (a^{-1/\ga} w,t, x)$ which scale space by $a$ (hence the entropy by $a^2$) and weights by $a^{-1/\ga}$ respectively. 
The random sets $\phi(\cP)$ and $\psi(\cP)$ are still
 two Poisson Point Processes with the same law, that is $\phi(\cP)\stackrel{(\dd)}{=} \psi(\cP)$.
 This implies that (recall the definition~\eqref{def:discrCont})
 \[
 \pi(a s) \overset{(\dd)}{=} a^{1/\alpha} \pi(s).
 \]
 Therefore, 
 \begin{equation}
 \label{eq:scalingAalpha2}
 \sup_{s\in \sD, \ent(s)<\infty} \big \{ \beta \pi(s) -a^2 \ent(s) \big\} \overset{(\dd)}{=}\sup_{s\in \sD, \ent(s)<\infty} \big \{  \beta a^{-1/\alpha} \pi(s) - \ent(s) \big \}.
 \end{equation}
  Consequently, 
  for any $\alpha\in (0,2)$, $a^2 \cT_{\gb/a^2} \stackrel{(\dd)}{=}  \cT_{\beta a^{-1/\alpha}}$. 
  In particular, for any $\beta>0$ it holds true that for $\ga>1/2$
  \begin{align}
  \cT_\gb \stackrel{(\dd)}{=}  \gb^{\frac{2\alpha}{2\alpha-1}} \cT_1\, .
  \end{align}

\subsection{Finite moments of $\cT_{\gb}$}
\label{sec:finitemoments}
We show that for $\alpha\in (1/2,2)$  $\bbE[(\cT_{\gb})^{\upsilon}]<\infty$ for any $\upsilon <\ga-1/2$, which readily implies that $\cT_\beta<\infty$ a.s.
For any interval $[c,d)$ with $0\le c< d$ we let
\begin{equation}
\label{def:Tcd}
\cT_\beta \big( [c,d) \big):= \sup\limits_{s\in \sD, \ent(s)\in [c,d)}\big\{\beta\pi(s)-\ent(s)\big\},
\end{equation}
and we observe that $\cT_\beta=\cT_\beta \big( [0,1) \big)\vee \displaystyle\sup_{k\ge 0} \cT_\beta \big( [2^k,2^{k+1} ) \big)$.
Moreover, as in \eqref{eq:scalingAalpha2} we have 
\begin{align}
\cT_\beta \big( [2^k,2^{k+1}) \big) & \overset{(\dd)}{=}\sup\limits_{s:\, \ent(s)\in [1,2)} \big\{\, 2^{\frac{k}{2\alpha}}\beta\pi(s)- 2^{k}\, \ent(s)\big\} \notag \\
& \le  2^{\frac{k}{2\ga}} \gb \sup_{s:\ent(s) \le 2} \pi(s)  - 2^k \, .
\label{Tk}
\end{align}
We show the following Lemma.
\begin{lemma}
\label{lem:moment}
For any $a <\ga$, we have that there is a constant $c_a>0$ such that for any $t >1$ we get
\[\bbP\Big( \sup_{s\in \sD, \ent(s) \le 2} \pi(s) > t \Big) \le c_a t^{-a} \, .\]
\end{lemma}

From this lemma and \eqref{Tk}, we get that for any $t \ge -1$ and any $k$ large enough so that $\gb^{-1} 2^{- \frac{k}{2\ga}}2^{-k}>2$, we get
\begin{align}
\bbP\Big(\cT_\beta \big( [2^k,2^{k+1}) \big) >t  \Big)
& \leq \bbP\Big( \sup_{s \in \sD , \ent(s) \le 2} \pi(s) > \gb^{-1} 2^{-\frac{k}{2\ga}} (t +2^k)\Big)\notag\\ 
&\le  c_a   \gb^{a} 2^{k  \tfrac{a}{2\ga} }   \big( t+ 2^k \big)^{-a} \, .
\label{tailTk}
\end{align}
Then, for any $t\ge 1$ and $a<\ga$, we get by a union bound that
\begin{align*}
\bbP\big( \cT_{\gb} >t\big) & \le \sum_{k=0}^{\infty} \bbP\Big(\cT_\beta \big( [2^k,2^{k+1}) \big) >t  \Big) \\
& \le c'_a 2^a \gb^{a} t^{-a}  \sum_{k = 0}^{\log_2 t}  2^{ k \tfrac{a}{2\ga}}     +  c'_a 2^a  \gb^{a}\sum_{k>\log_2 t} 2^{ -a k \big( 1-\tfrac{1}{2\ga} \big) }  \\
& \le c''_a  \gb^a t^{-a} t^{\frac{a}{2\ga}} + c''_a  t^{-a \big( 1-\tfrac{1}{2\ga} \big) } \le 2c''_a \gb^a  t^{-a \big( 1-\tfrac{1}{2\ga} \big) } \, ,
\end{align*}
where we used that $t+2^k \ge t/2$ if $k\le \log_2 t$, and $t+2^k \ge 2^k/2$ if $k> \log_2 t$. For the second sum we also used that $1-\tfrac{1}{2\ga} >0$ when $\ga>1/2$. 
In particular, this shows that for any $\gd>0$, there is some constant $c_{\gd,\gb}>0$ such that for any $t\ge 1$
\begin{equation}\label{bondTbeta}
\bbP\big( \cT_{\gb} >t\big) \le c_{\gd,\gb} t^{-(\ga-\frac12) +\gd} \, ,
\end{equation}
which proves that $\bbE[(\cT_{\gb})^{\upsilon}]<\infty$ for any $\upsilon <\ga-1/2$.

\begin{proof}[Proof of Lemma \ref{lem:moment}]

Let us recall that $\ent(s)\leq 2$ implies that $ \max |s|\leq 2$. Therefore we can restrict our Poisson Point Process to $\mathbb R_+\times [0,1]\times [-2,2]$. 
 In this case (cf. Section \ref{contLimitHT} below) we rewrite a realization of the Poisson Point Process by using its ordered statistic.
We introduce $(Y_i)_{i\in\mathbb N}$ be an i.i.d.\ sequence of uniform random variables on $[0,1]\times [-2,2]$ and $(M_i)_{i\in\mathbb N}$ be a random sequence independent of $(Y_i)_{i\in \mathbb N}$ defined by  $M_i = 4^{1/\ga} (\cE_1 + \cdots +\cE_i)^{-1/\ga} $ with $(\cE_j)_{j\ge 1}$ an i.i.d.\ sequence of ${\rm Exp}(1)$ random variables.
 In such a way $\cP\overset{(\dd )}{=}(M_i,Y_i)_{i\in\mathbb N}$ and $\pi(s)= \sum_{i=1}^\infty  M_i \ind_{\{Y_i \in s\}}$.

The proof is then a consequence of Theorem \ref{thm:generalLPP} (with $B=1$), which allows to use the same ideas as in \cite[Proposition 3.3]{HM07} -- we develop the argument used in \cite{HM07} in a more robust way, which makes it easier to adapt to the discrete setting. Using the notations introduced in Section \ref{sec:ELPP}, for any $i\ge 0$, we denote $\boldsymbol \gU_i = \{Y_1, \ldots, Y_i\}$ ($\boldsymbol \gU_0 =\emptyset$), and let $\Delta_i = \Delta_i(s) = s\cap \boldsymbol \gU_i $ be the set of the $i$ largest weights collected by $s$. The E-LPP can be written here as $\mathcal L_i^{(2)} := \max_{s: \ent(s) \le 2} |\Delta_i(s)|$ -- we drop here the dependence on $t,x$.

Using that $M_i$ is a non-increasing sequence, we write
\begin{equation}
\label{eq:split-sup}
\pi(s) = \sum_{j=0}^{\infty} \sum_{i=2^{j}}^{2^{j+1}-1} M_i \ind_{\{Y_i\in s\}} 
\le  \sum_{j=0}^{\infty} M_{2^j} \mathcal{L}_{2^{j+1}}^{(2)}\, .
\end{equation}

Then, we fix some $\gd>0$ such that $1/\ga-1/2 >2\gd$, and we let $C=\sum_{j=0}^{\infty} 2^{j(1/2-1/\ga+2\gd)}$: we obtain via a union bound  that
\begin{align}
\bbP\Big( &\sup_{\ent(s)\le 2} \pi(s) >t  \Big) \le \sum_{j=0}^{\infty} \bbP\Big( M_{2^j}  \mathcal{L}_{2^{j+1}}^{(2)} > \frac1C \, t\, 2^{j(1/2-1/\ga+2\gd)} \Big) \nonumber \\
&\le   \sum_{j=0}^{\infty}  \Big[ \bbP\Big( \mathcal{L}_{2^{j+1}}^{(2)} > C' \log t \, (2^{j+1})^{1/2+\gd)} \Big) + \bbP\Big( M_{2^j} > C'' \, \frac{t}{\log t} (2^{j})^{-1/\ga+\gd} \Big)  \Big]\, .
\label{split-sup-unionbound}
\end{align}
Here $C'$ is a constant that we choose large in a moment, and $C''$ is a constant depending on $C,C'$ - we also work with $t\ge 2$ for simplicity.

For the first probability in the sum, we obtain from Theorem \ref{thm:generalLPP}-(i) that
 provided $C' (\log t)  2^{j\gd}\ge 2 C_0^{1/2}$
\[
\bbP\Big( \mathcal{L}_{2^{j+1}}^{(2)} > C' \log t \, (2^{j+1})^{1/2+\gd)} \Big) \le \Big(\frac12\Big)^{C' (\log t)  2^{j\gd}} \le t^{- \log 2\,  C' 2^{j\gd}}\, .
\]
Hence, for $t$ sufficiently large we get that
\begin{equation}\label{L2split1}
 \sum_{j=0}^{\infty}   \bbP\Big( \mathcal{L}_{2^{j+1}}^{(2)} > C' \log t \, (2^{j+1})^{1/2+\gd)} \Big)  \le  c t^{- C'\log 2} \le c t^{-a}
\end{equation}
provided that we fixed $C'$ large.

For the second probability in the sum, recall that $M_i \stackrel{\rm (d)}{=} 4^{1/\ga} {\rm Gamma}(i)^{-1/\ga}$, so that for any $a<\ga$, $\bbE[( i^{1/\ga} M_i )^{a}]$ is bounded by a constant independent of $i$.
Therefore, Markov's inequality gives that
\[\bbP\Big( M_{2^j} > C'' \, \frac{t}{\log t} (2^{j})^{-1/\ga+\gd} \Big) \le c (\log t)^a t^{-a} (2^{j})^{-a\gd}\,  ,\]
so that 
\begin{equation}\label{M2split1}
 \sum_{j=0}^{\infty} \bbP\Big( M_{2^j} > C'' \, \frac{t}{\log t} (2^{j})^{-1/\ga+\gd} \Big) \le c (\log t)^a t^{-a} \, .
\end{equation}
Plugging \eqref{L2split1} and \eqref{M2split1} into \eqref{split-sup-unionbound}, we obtain that for any $a'<a <\ga$ there are constants $c>0$ such that for any $t\ge 2$
\[\bbP\Big( \sup_{\ent(s)\le 2} \pi(s) >t  \Big) \le 2 c (\log t)^a t^{-a} \le c' t^{-a'} \, ,\]
which concludes the proof.
\end{proof}

\subsection{Continuity of $\gb \mapsto \cT_{\gb}$}
\label{sec:continuity}

An obvious and crucial fact that we use along the way is that  for any realization of $\cP$, $\gb\mapsto \cT_{\gb}$ is non-decreasing.

\subsubsection{Left-continuity}
Let us first show that $\gb \mapsto \cT_{\gb}$ is left-continuous, since it is less technical.
Fix $\gep>0$.  For any $\gb$, there exists a path $s^{(\gep)}_{\gb}$ with $\pi(s_{\gb}^{(\gep)}) <\infty$ such that $\cT_{\gb} \le \beta \pi(s^{(\gep)}_{\gb}) -\ent(s^{(\gep)}_{\gb}) +\gep$.  Using this path $s_{\gb}^{(\gep)}$, we then  simply write that for any $\gd>0$
\[\cT_{\gb} \ge \cT _{\gb-\gd} \ge (\gb-\gd) \pi(s^{(\gep)}_{\gb}) - \ent(s^{(\gep)}_{\gb}) \, . \]
Letting $\gd\downarrow 0$, we get that the right hand side converges to $\gb \pi(s^{(\gep)}_{\gb}) - \ent(s^{(\gep)}_{\gb}) \ge \cT_{\gb} -\gep$. 
Since $\gep$ is arbitrary, one concludes that $\lim_{\gd \uparrow 0} \cT_{\gb-\gd} =\cT_{\gb}$, that is $\gb \mapsto \cT_{\gb}$ is left-continuous.

\subsubsection{Right-continuity}
It remains to prove that a.s. $\gb \mapsto \cT_{\gb}$ is right-continuous. 
We prove a preliminary result.
\begin{lemma}\label{lem:relsupK0}
For any $K>0$, $\bbP$-a.s.\ there exists $B_0>0$ such that for any $0\le \gb \le K$
\begin{equation}\label{relsupK0}
\cT_{\gb} = \cT_{\gb}\big( [0,B_0] \big)\,  ,
\end{equation} 
where $\cT_{\gb}\big( [0,B_0] \big)$  is defined in \eqref{def:Tcd}.
\end{lemma}
\begin{proof}
Let us recall that $\cT_\beta=\cT_\beta \big( [0,1) \big)\vee \displaystyle\sup_{k\ge 0} \cT_\beta \big( [2^k,2^{k+1} ) \big)$. 
Using \eqref{tailTk} with $t=-1$, for any $a<\alpha$ we have that 
\[
\bbP\Big(\cT_\beta \big( [2^k,2^{k+1}) \big) >-1  \Big)
\le  c_a  \gb^{a} 2^{k  \tfrac{a}{2\ga} }   \big( 2^k-1 \big)^{-a} \le c_{a,K} \, 2^{k (\tfrac{1}{2\ga}-1) } .
\]
Since $\tfrac{1}{2\ga}-1<0$, by Borel-Cantelli lemma we obtain that
  $\bbP$-a.s. there exists $k_0>0$ such that $\cT_\beta \big( [2^k,2^{k+1})) \le -1$ for all $k\ge k_0$. This concludes the proof. 
\end{proof}
Then, since we now consider paths with entropy bounded by $B_0$,
we can restrict the Poisson Point Process $\cP$ to  $\mathbb R_+\times [0,1]\times [-\sqrt{2B_0},\sqrt{2B_0}]$. 
 In this case we write a realization of the Poisson Point Process by using its ordered statistic. More precisely we introduce $M_i:= (8B_0)^{1/2\ga} (\cE_1+\cdots+\cE_i)^{-1/\ga}$, where $(\cE_i)_{i\in \mathbb N}$ is an i.i.d.\ sequence of exponential of mean~$1$ and $(Y_i)_{i\in \mathbb N}$ is a i.i.d.\ sequence of uniform random variables on $[0,1]\times [-\sqrt{2B_0},\sqrt{2B_0}]$, independent of $(\cE_i)_{i\in \mathbb N}$.
 Then, $\cP\overset{(\dd )}{=}(M_i,Y_i)_{i\in\mathbb N}$ and $\pi(s)=
 \sum_{i=1}^\infty  M_i \ind_{\{Y_i \in s\}}$.

For any $\ell\in \mathbb N$, we let $\pi^{(\ell)}:=\sum_{i=1}^\ell M_i\ind_{Y_i\in s}$ be the ``truncated'' energy of a path: we can write for any $\gb < K $, and any $\gd>0$ such that $\gb+\gd \le K$
\begin{align*}
 \cT_{\gb+\gd} = \cT_{\gb+\gd}\big( [0,B_0] \big) \le \sup_{s\in \sD, \ent(s) \le B_0} &\big\{ (\gb+\gd) \pi^{(\ell)}(s) - \ent(s) \big\}\\
 & + (\gb+\gd) \sup_{s\in \sD, \ent(s) \le B_0} \big | \pi(s) -\pi^{(\ell)}(s) \big| \, .
\end{align*}
Then, we show that
\begin{equation}
\label{truncatedconv}
\max_{s\in \sD, \ent(s) \le B_0}\big|\pi(s)-\pi^{(\ell)}(s)\big|\xrightarrow[\ell\to\infty]{a.s.}0 \, .
\end{equation}
Hence, for any fixed $\gep$, we can a.s.\ choose some $\ell_{\gep}$ such that  for any $\gb < K$ and any $\gd>0$ with $\gb+\gd \le K$
\[\cT_{\gb} \le \cT_{\gb+\gd}  \le  \sup_{s\in \sD, \ent(s) \le B_0} \big\{ (\gb+\gd) \pi^{(\ell)}(s) - \ent(s) \big\} + K \gep \, .\]
Then, letting $\gd\downarrow 0$, and since the supremum can now be reduced to a finite set (we consider only $\ell$ points), we get that for any $\gb <K$
\[\cT_{\gb} \le \lim_{\gd\downarrow 0} \cT_{\gb+\gd} \le \sup_{s\in \sD, \ent(s) \le B_0} \big\{ \gb \pi^{(\ell)}(s) - \ent(s) \big\}  + \gep \le \cT_{\gb} + \gep\, . \]
Since $\gep$ is arbitrary, this shows that $\lim_{\gd\downarrow 0} \cT_{\gb+\gd} =\cT_{\gb}$ a.s., that is $\gb\mapsto \cT_{\gb}$ is right-continuous.

It remains to prove \eqref{truncatedconv}. For any $i\in \mathbb N$ we consider $\boldsymbol \gU_i=\{Y_1,\dots,Y_i\}$ and for any given path $s$ we define $\Delta_i = \Delta_i(s) = s\cap \boldsymbol \gU_i $ the set of the $i$ largest weights collected by $s$. Then, let $\cL_i^{(B_0)} = \sup_{s\in \mathscr{D}_{B_0}} |\Delta_i(s)|$, as introduced 
in \eqref{def:LcontLdis}. 
Realizing that $\ind_{\{Y_i\in s\}} = |\Delta_i (s)| - |\Delta_{i-1}(s)|$, and
integrating by parts (as done in \cite{HM07}), we obtain for any $s\in \sD_{B_0}$
 \begin{align}
  \pi(s)-\pi^{(\ell)}(s) &= \sum_{i> \ell } M_i \ind_{\{Y_i\in s\}} = \lim_{n\to\infty} \sum_{i= \ell+1 }^n M_i \big( |\Delta_i| - |\Delta_{i-1}| \big)\nonumber\\
  &=\lim_{n\to\infty} \bigg(\sum_{i=\ell+ 1}^{n-1} |\Delta_i|(M_{i}-M_{i+1})+M_n |\Delta_n|  -  M_{\ell} |\Delta_{\ell}|\bigg)  \nonumber\\
  &\le  \sum_{i=\ell+1}^{\infty} \cL_i^{(B_0)} (M_{i}-M_{i+1})+\limsup_{n\to\infty} M_n \cL_n^{(B_0)} .\label{sumrest}
 \end{align}
At this stage, the law of large numbers gives that  $ \lim_{n\to\infty} n^{1/\ga} M_n = (8B_0)^{1/2\ga}$ a.s., and Corollary~\ref{cor:LPP} gives that $\limsup_{n\to\infty} n^{-1/2} \cL_n^{(B_0)} <+\infty$ a.s. Since $\ga<2$, we therefore conclude that $ \limsup_{n\to\infty} M_n \cL_n^{(B_0)}= 0$ a.s.

We let $U_\ell := \sum_{i >\ell} \cL_i^{(B_0)} (M_i-M_{i-1}) $.
We are going to show that there exists some $\ell_0$ such that $\sum_{i >\ell_0} \cL_i^{(B)} (M_i-M_{i-1}) <\infty$ a.s., and thus $\lim_{\ell\to\infty }U_\ell = 0$ a.s. 
We show that $\bbE [U_{\ell_0}^2]$ is finite for $\ell_0$ large enough. For any $\gep>0$, by Cauchy-Schwarz inequality we have that
\begin{align*}
U_{\ell_0} \le \Big(\sum_{i> \ell_0} \big(i^{-\frac12 -\gep} \big)^2 \Big)^{1/2} \Big( \sum_{i>\ell_0} \big( i^{-\frac12 +\gep} \cL_i^{(B_0)} (M_i - M_{i+1}) \big)^2\Big)^{1/2}\, .
\end{align*}
Then, we get that for $\ell_0$ large enough
\begin{align*}
\bbE [U_{\ell_0}^2] &\le C \sum_{i > \ell_0} i^{1+2\gep} \bbE\big[ (\cL_i^{(B)})^2\big] \bbE \big[ (M_i -M_{i-1})^2 \big]\\
&\le C'_{B_0} \sum_{i >\ell_0} i^{1+2\gep} \times i \times i^{-2-2/\ga} <+\infty \, .
\end{align*}
Here, we used Corollary \ref{cor:LPP} and a straightforward calculation that gives $\bbE \big[ (M_i -M_{i-1})^2 \big] \le c i^{-2-2/\ga}$ for $i$ large enough (see for instance Equation  (7.2)  in \cite{HM07}). Provided $\gep$ is small enough so that $2\gep - 2/\ga <-1$ we obtain that $\bbE [ U_{\ell_0}^2]<\infty$.
\qed

\subsection{Existence and uniqueness of the maximizer}

As a consequence of Lem\-ma~\ref{lem:relsupK0}, to show that the supremum is attained and is unique  in \eqref{def:T}, it is enough to prove the
following result.
\begin{lemma}\label{lemmasupatt}
For a.e. realization of $\cP$ and for any $B>0$ we have that
$$
s_\beta^*(B)=\argmax_{s \in \sD_B } \big\{ \beta \pi(s) - \ent(s)  \big\}\,
$$
exists, and it is unique. Here, we defined  $\mathscr{D}_{B}:= \{s\in \mathscr{D} \colon \ent(s)\leq B \}$.
\end{lemma}
\begin{proof}
Our first step is to show that $\mathscr{D}_{B}$ is compact for the uniform norm $\| \cdot \|_\infty$. 
Let us observe that for any $s:[0,1]\to \mathbb R$, the condition $\ent(s)\le B$ implies that
\begin{equation*}
|s(x)-s(y)|\leq \int_y^x |s'(t)| d t \leq (2B)^{1/2} |x-y|^{1/2},\qquad \forall\, x,y\in [0,1],
\end{equation*}
so that $s$ belongs to the H\"older Space $C^{1/2}([0,1])$.
Hence, $\mathscr{D}_{B}$ is included in $ C^{1/2}([0,1])$ which is compact for the uniform norm $\| \cdot \|_\infty$ by the Ascoli-Arzel\`a theorem. We therefore only need to show that $\mathscr{D}_{B}$ is closed for the uniform norm $\| \cdot \|_\infty$.

For this purpose we consider a convergent sequence $s_n$ and we 
denote by $\mathbf{s}$ its limit. 
Since $\ent(s_n)=\frac12 \|s_n'\|_{L^2}^2$ for all $n$, we have that $(s_n')_{n\in \mathbb N}$ belongs to the (closed) ball of 
radius $(2B)^{1/2}$ of $L^2([0,1])$. 
By Banach–Alaoglu theorem, the sequence $(s_n')_{n\in \mathbb N}$  contains a weakly convergent subsequence. This means that there exist $n_k$ and $s^{\star}$ such that
$$
\int_0^1 \phi(x) s_{n_k}'(x) \dd x \xrightarrow[k\to\infty]{} \int_0^1 \phi(s)  s^{\star}(x) d x, \, \qquad \forall \phi\in L^2([0,1]).
$$
By uniqueness of the limit (and taking $\phi(x)=\ind_{\{[0,y]\}}(x)$), this relation implies that $\mathbf{s}(y)=\int_0^y s^\star (x) d x$, that is $\mathbf{s}'= s^\star$ almost everywhere. Since the $L^2$ norm is  weakly lower semi-continuous by the Hahn-Banach theorem -- that is $\|s^\star\|_{L^2}\le \liminf_{k\to\infty}\|s_{n_k}'\|_{L^2}$ -- we obtain that $\mathbf{s}\in \mathscr{D}_B$, so $\mathscr{D}_B$ is closed. 
 As a by-product of this argument we also have that the entropy function $s \mapsto \ent(s)$ is lower semi-continuous on $(\mathscr D_B, \|\cdot\|_{\infty})$.

\subsubsection{Existence of the maximizer}
Since $\mathscr{D}_B$ is compact, the existence of the maximizer comes from the fact that the function
\begin{equation}
\label{tbeta}
t_\beta(s):= \gb \pi(s)-\ent(s)
\end{equation}
is upper semi-continuous, thanks to the extreme value theorem tells. Since we have already shown that $s \mapsto \ent(s)$ is lower semi-continuous, we only need to prove the following. 

\begin{lemma}\label{piIsUP}
For a.e. realization of $\cP$ and for any $B>0$ the function $s\mapsto \pi(s)$ is upper semi-continuous on $(\mathscr D_B, \|\cdot\|_{\infty})$. 
\end{lemma}
\begin{proof}
We recall that if $s\in \mathscr{D}_B$ then $ \max_{t\in [0,1]} |s(t)|\leq \sqrt{2B}$. Therefore, using the same notations as above, we can write a realization of the Poisson Point Process $\cP$ by using its ordered statistic: $\cP=(M_i,Y_i)_{i\in\mathbb N}$, $\pi(s)=
 \sum_{i=1}^\infty  M_i \ind_{\{Y_i \in s\}}$, and recall that for any $\ell\in \mathbb N$ we let $\pi^{(\ell)}:=\sum_{i=1}^\ell M_i\ind_{\{Y_i\in s\}}$. 
Thanks to \eqref{truncatedconv}, we only need to prove that  for any fixed $\ell\in \mathbb N$ the function $s\mapsto \pi^{(\ell)}(s)$ is upper semi-continuous: then $\pi(s)$, as the uniform limit of $\pi^{(\ell)}$, is still upper semi-continuous.

For any $s\in \mathscr D_B$ we let $\iota_s :=\boldsymbol \gU_\ell\setminus \{s\cap \boldsymbol \gU_\ell\}$
be the set of all points of $\boldsymbol \gU_\ell=\{Y_1,\dots, Y_\ell\}$ that are not in $s$. 
Since there are finitely many points, we realize that there exists $\eta=\eta(s,\ell)>0$ such that $d_H(\iota_s,\text{graph}(s))>\eta$, with $d_H$ is the Hausdorff distance. 

Given $s\in \mathscr D_B$, we consider a sequence $(s_n)_n, \, s_n\in \mathscr D_B$ that converges to $s$, $\lim_{n\to\infty}\|s_n-s\|_\infty=0$. 
We observe that whenever $\|s_n-s\|_{\infty}\le \eta/2$, we have that $d_H(\iota_s,\text{graph}(s_n))>\eta/2$. This means that for $n$ large enough
\[
\{s_n\cap \boldsymbol \gU_\ell\}\subset\{s\cap \boldsymbol \gU_\ell\}\, ,
\]
which implies that $\pi^{(\ell)}(s)\ge \limsup_{n\to\infty}\pi^{(\ell)}(s_n)$.
\end{proof}

\subsubsection{Uniquenes of the maximizer}  The strategy is very similar to the one used in \cite[Lemma 4.1]{AL11} or \cite[Lemma 4.2]{HM07}. 
For any $s\in \mathscr D_B$, we let $I(s):=\{s\cap \boldsymbol \gU_\infty\}$, where we $\boldsymbol \gU_\infty=\{Y_i, \, i\in\mathbb N\}$.

Let us assume that we have two maximizers $s_1\neq s_2$. 
 Since $\boldsymbol \gU_\infty$ is dense in $[0,1]\times [-\sqrt{2B},\sqrt{2B}]$ we have that $I(s_1)\neq I(s_2)$. 
In particular, there exists $i_0$ such that $Y_{i_0}\in I(s_1)$ and $Y_{i_0}\not\in I(s_2)$, and since $s_1$ and $s_2$ are two maximizers of \eqref{tbeta} it means
\[
\max_{s\colon Y_{i_0}\in I(s)} t_\beta(s)=\max_{s\colon Y_{i_0}\not\in I(s)} t_\beta(s).
\]
This implies that 
\begin{equation}
\label{contradiction}
\beta M_{i_0}= \max_{s\colon Y_{i_0}\not\in I(s)} t_\beta(s)-\max_{s\colon Y_{i_0}\in I(s)}\bigg\{\gb \sum_{ j, j\neq i_0} M_j \ind_{\{Y_j\in s\}} -\ent(s)\bigg\} \, .
\end{equation}
Conditioning on $(Y_j)_{j\in \mathbb N}$ and $(M_j)_{j\in \mathbb N, j\neq i_0}$ we have that the l.h.s. has a continuous distribution -- the distribution of $M_{i_0}^{-\alpha}$
conditional on $(Y_j)_{j\in \mathbb N}$ and $(M_j)_{j\in \mathbb N, j\neq i_0}$ is uniform on the interval $[M_{i_0-1}^{-\ga}, M_{i_0+1}^{-\ga}]$~--, while the r.h.s. is a constant -- it is independent of $M_{i_0}$. Therefore the event \eqref{contradiction} has zero probability, and by sigma sub-additivity we get that $\bbP\big(I(s_1)\neq I(s_2)\big)=0$, which contradicts the existence of two distinct maximizers.
\end{proof}

\section{Proof of Proposition \ref{thm:discreteELPP} and Theorem \ref{prop:ConvVP}}

\label{sec:proofdisc}
Let us state right away a lemma that will prove to be useful in the rest of the paper. 
\begin{lemma}
\label{lem:Mell}
For any $\eta>0$, there exists a constant $c$ such that, for any $t>1$ and any $\ell \le nh$,  we have
\[\bbP\Big( M_\ell^{(n,h)} > t m( \tfrac{nh}{\ell})  \Big) \le (ct)^{ -(1-\eta)\ga \ell} \, . \]
\end{lemma}
\begin{proof}
We simply write that by a union bound
\[
\bbP\Big( M_r^{(n,h)} > t m( \tfrac{nh}{r} )  \Big)  \le \binom{nh}{r} \bbP\Big(\go _1 > t m( \tfrac{nh}{r}) \Big)^r 
\le   \Big( c \frac{nh}{r}  \bbP\big(\go _1 > t m(\tfrac{nh}{r})  \big) \Big)^r\, .
\]
Then, since $\bbP(\go_1 >x)$ is regularly varying with exponent $-\ga$, Potter's bound (cf. \cite{BGT89}) gives that there is a constant $c_{\eta}$ such that for any $t\ge 1$
\[\bbP\big(\go _1 > t m(\tfrac{nh}{r})  \big) \le c_{\eta} t^{-(1-\eta)\ga} \bbP\big(\go _1 >  m(\tfrac{nh}{r})  \big) = c_{\eta} t^{-(1-\eta)\ga }\, \frac{nh}{r} \, ,\]
where we used the definition of $m(\cdot)$ in the last identity. This concludes the proof.
\end{proof}

\subsection{Continuum limit of the ordered statistic}
\label{contLimitHT}

For any $q>0$  let $\Lambda_{n,qh} = \llbracket 1,n \rrbracket \times \llbracket -qh,qh \rrbracket$ and let $(M_r^{(n,qh)},Y_r^{(n,qh)})_{r=1}^{|\Lambda_{n,qh}|}$ be the ordered statistic in that box, cf. \eqref{eq:ordstat}.
If we rescale $\Lambda_{n,qh}$ by $n\times h$, and we let $(\tilde Y_r^{(n,qh)})_{r=1}^{|\Lambda_{n,h}|}$ be the rescaled permutation, \textit{i.e.}\ a random permutation of the points of the set $([0,1]\times [-q,q])\cap (\frac{\mathbb N}{n}\times \frac{\mathbb Z}{h})$.
Then for any fixed $\ell\in \mathbb N$, 
\begin{equation}
\label{convYn}
\big(\tilde Y_1^{(n,qh)},\dots,\tilde Y_\ell^{(n,qh)} \big) \overset{(\dd)}{\rightarrow}\big( Y_1,\dots, Y_\ell \big), \quad \text{as}\quad n,h\to\infty,
\end{equation}
where $(Y_i)_{i\in\mathbb N}$ is an i.i.d.\ sequence of uniform random variables on $[0,1]\times [-q,q]$.
For the continuum limit for the weight sequence $(M_r^{(n,qh)})_{r=1}^{|\Lambda_{n,qh}|}$, we use some basic facts of 
the classical extreme value theory (see e.g., \cite{R87}), that is
for all $\ell\in\mathbb N$,
\begin{equation}
\label{convWeight}
\Big(\tilde M_i^{(n,qh)}:=\frac{M_i^{(n,qh)}}{m(nh)}, \, i=1,\dots, \ell\Big)
\overset{(\dd)}{\rightarrow} \big(M_i,\, i=1,\dots, \ell \big),
\end{equation}
where $(M_i)_{i\in\mathbb N}$ is the \emph{continuum weight sequence}. The sequence $(M_i)_{i\ge 1}$ can be defined as $M_i:=(2q)^{1/\ga}(\cE_1+\dots+\cE_i)^{-\frac{1}{\alpha}}$, where $(\cE_i)_{i\in \mathbb N}$ is an i.i.d.\ sequence of 
exponential random variables of mean $1$, independent of the~$Y_i$'s. 

In such a way $(M_i,Y_i)_{i\in\mathbb N}$ is the ordered statistic associated with a realization of
a Poisson Point Process on $[0,\infty)\times [0,1]\times[-q,q]$ 
of intensity $\mu(\dd w, \dd t,\dd x)=\frac{\alpha}{2} w^{-\alpha-1}\ind_{\{w>0\}} \dd w \dd t \dd x$.

\subsection{Proof of Theorem \ref{prop:ConvVP}}
For any $q>0$, 
we consider the Poisson Point Process restricted to $[0,1]\times [-q,q]$, and we label its elements according to its ordered statistic $(M_i, Y_i)_{i\in \mathbb N}$. 
For any $\Delta\subset[0,1]\times [-q,q]$  we define $\pi^{(\ell)}(\Delta)=\sum_{i=1}^\ell M_i\ind_{\{Y_i\in \Delta\}}$ and $\pi^{(>\ell)}(\Delta):=\pi(\Delta)-\pi^{(\ell)}(\Delta)$. 
In analogy with the discrete setting (cf. \eqref{def:discrELPPell}), we define
\begin{equation}\label{def:TAL}
\begin{split}
&\cT_{\nu,q}^{(>\ell)} = \sup_{s\in \sM_q} \big\{ \nu \pi^{(>\ell)}(s) - \ent(s) \big\}\, ,
\\
&  \cT_{\nu,q}^{(\ell)} = \sup_{s\in \sM_q} \big\{ \nu \pi^{(\ell)}(s) - \ent(s) \big\}\,.
\end{split}
\end{equation}
We first show the convergence \eqref{conv:largeweights} of the large-weights variational problem, before we prove~\eqref{def:TA}.

\emph{Convergence of the large weights.}
Note that the maximum of $T_{n,qh}^{\gb_{n,h}, (\ell)}$ and $\cT_{\nu,q}^{(\ell)}$ are achieved on $\gU_\ell=\gU_\ell(q)$ and $\boldsymbol \gU_\ell=\boldsymbol \gU_\ell(q)$ respectively, that is
\begin{equation}
\label{conv:Tl}
\begin{split}
&T_{n,qh}^{\gb_{n,h}, (\ell)}=\max_{\Delta\subset \gU_\ell}\big\{  \beta_{n,h} \Omega_{n,h}^{(\ell)} (\Delta) - \ent(\Delta)  \big\}\, ,
\\
&\cT_{\nu,q}^{(\ell)} = \sup_{\Delta\subset \boldsymbol \gU_\ell} \big\{ \nu \pi^{(\ell)}(\Delta) - \ent(\Delta) \big\} \, ,
\end{split}
\end{equation}
where $\gU_\ell(q)$ (resp. $\boldsymbol \gU_\ell(q)$) is the set of the locations of the $\ell$ largest weights inside $\Lambda_{n,qh}$ (resp. $\boldsymbol \Lambda_{1,q}$). 
Since we have only a finite number of points, the convergence~\eqref{conv:largeweights} is a consequence of \eqref{convYn} and \eqref{convWeight} and the Skorokhod representation theorem.

\emph{Restriction to the large weights.} 
To show the convergence \eqref{def:TA}, it is therefore enough to control the contribution of the large weights.
Let $\delta>0$ such that $\frac{1}{\alpha}-\frac12>\delta$. Using Potter's bound (cf. \cite{BGT89}) we have that 
\[
\big(\gb_{n,h} m(nh/\ell)\big)^{\frac43} \Big(\frac{\ell^2 n}{h^2}\Big)^{\frac13} \le c \frac{h^2}{n} \ell^{-\frac{4}{3}(\frac{1}{\alpha}-\frac12-\delta)}.
\]
Plugging it into \eqref{bigthanell} and taking $b=b_{\ell,\epsilon}:=\epsilon\, \ell^{\frac{4}{3}(\frac{1}{\alpha}-\frac12+\delta)}$, 
we obtain that 
\begin{equation}\label{TbigELLrest}
\bbP \Big(  \frac{n}{h^2} T_{n,qh}^{\gb_{n,h}, (>\ell)}\ge \epsilon\Big) \le  
c' b_{\ell,\epsilon}^{-\ga \ell/4} + e^{- c' b_{\ell,\epsilon}^{1/4}} \, \overset{\ell\to\infty}{\longrightarrow} \, 0,
\end{equation}
uniformly on $n, h$. Combined with \eqref{conv:largeweights} and the first part of \eqref{Tconverge}, this gives the convergence \eqref{def:TA}.

\emph{Proof of \eqref{Tconverge}.}

This is a simple consequence of the monotonicity of $\ell \mapsto \cT_{\nu,q}^{(\ell)}$ and of $q\mapsto \cT_{\nu,q}$ (together with the fact that $\cT_{\nu}$ is well defined).

\subsection{Proof of Proposition \ref{thm:discreteELPP}}\label{proofthmELPPd}
Let us first focus on $T_{n,h}^{\gb_{n,h},(\ell)} $.
As in \eqref{def:Tcd} in the continuous setting,  we introduce, for any interval $[c,d)$,
\begin{equation}
\label{def:Tdisc-cd}
T_{n,h}^{\gb_{n,h},(\ell)} \big( [c,d) \big):= \max_{ \Delta \subset \Lambda_{n,h}, \ent(\Delta) \in [c,d) } \big\{  \gb_{n,h} \Omega_{n,h}^{(\ell)} (\Delta) - \ent(\Delta)  \big\} \, .
\end{equation}
Then, we realize that for any $d>0$
\[
T_{n,h}^{\gb_{n,h},(\ell)}  =  T_{n,h}^{\gb_{n,h},(\ell)} \big( [0,d) \big) \vee \sup_{k \ge 1 } T_{n,h}^{\gb_{n,h},(\ell)} \big( [2^{k-1} d,2^{k} d) \big)   \, . 
\]
Using that 
\begin{align*}
&T_{n,h}^{\gb_{n,h},(\ell)} \big( [2^{k-1} d,2^{k}d) \big) \le \gb_{n,h} \sup_{ \Delta\, : \,\ent(\Delta)  \le 2^{k} d}   \Omega_{n,h}^{(\ell)} (\Delta)    - 2^{k-1} d , \quad \text{for}\, k\ge 1, \\
&T_{n,h}^{\gb_{n,h},(\ell)} \big( [0,d) \big) \le \gb_{n,h} \sup_{ \Delta\, : \,\ent(\Delta)  \le  d}   \Omega_{n,h}^{(\ell)} (\Delta),
\end{align*}
with the choice $d=  b \, \hat\gb $ and { $\hat \gb := (\gb_{n,h} m(nh))^{4/3}( n/h^2)^{1/3}$, } a union bound gives that
\begin{align}
\label{Tnh-large}
\bbP\Big( T_{n,h}^{\gb_{n,h},(\ell)}  & \ge b \hat \beta 
 \Big)   \le   \sum_{k\ge  0 } \bbP\bigg(  \gb_{n,h}\sup_{ \Delta\, : \,\ent(\Delta)  \le 2^{k}   b  \hat \gb }   \Omega_{n,h}^{(\ell)} (\Delta)  \ge 2^{k-1}  b  \hat \gb  \bigg) \notag \\
  &   \le   \sum_{k\ge  0 } \bbP\bigg( \sup_{ \Delta\, : \,\ent(\Delta)  \le 2^{k}   b  \hat \gb }   \Omega_{n,h}^{(\ell)} (\Delta)  \ge 2^{k-1}  b \,  m(nh) \big( \hat\gb n/h^2 \big)^{1/4}  \bigg)\, ,
\end{align}
where we use that  $\hat\gb$ satisfies the equation $\hat \gb = \gb_{n,h} m(nh) (\hat \gb n/h^2)^{1/4}$.

We then need the following lemma, analogous to Lemma \ref{lem:moment}.
\begin{lemma}
\label{lem:Omeg}
For any $a<\alpha$, there exists a constant $c$ such that for any $B\ge 1$, $n,h\ge 1$, and any $t>1$
\[\bbP\Big( \sup_{ \Delta\, : \,\ent(\Delta)  \le B }   \Omega_{n,h}^{(\ell)} (\Delta)  \ge  t \times m(nh) \big( B n/h^2 \big)^{1/4}\Big) \le c t^{-a}\]
\end{lemma}
Applying this lemma in \eqref{Tnh-large} (with $B=2^k b \hat \gb $, $t= 2^{3k/4-1}b^{3/4}$), we get that for any $k\ge 0$  
\[\bbP\Big( \sup_{ \Delta\, : \,\ent(\Delta)  \le 2^k b \hat \gb }    \Omega_{n,h}^{(\ell)} (\Delta)  \ge 2^{k-1}  b   m(nh) 
(\hat\gb {n/h^2})^{1/4} \Big) \le c (2^k b)^{-3a/4} \, ,\]
so that summing over $k$ in \eqref{Tnh-large}, we get Proposition \ref{thm:discreteELPP}.

\begin{proof}[Proof of Lemma \ref{lem:Omeg}]
We mimic here the proof of Lemma~\ref{lem:moment}, but we need to keep  the dependence on the parameters $n,h,B$.
For $i\ge 0$, we denote $\gU_i:= \{ Y_1^{(n,h)}, \ldots, Y_{i}^{(n,h)}\}$ with the $Y_j^{(n,h)}$ introduced in Section \ref{sec:OrderedStat} ($\gU_0=\emptyset$), and for any $\Delta$ we let $\Delta_i := \Delta \cap \gU_i$ be the restriction of $\Delta$ to the $i$ largest weights. 
As in \eqref{eq:split-sup}, we can write
\begin{equation}
\label{decomposeOmega}
\frac{1}{m(nh) (Bn/h^2)^{1/4}}  \times \sup_{\Delta: \ent(\Delta) \le B}\Omega_{n,h}^{(\ell)}(\Delta) \le  \sum_{j=0}^{\log_2 \ell} \tilde M_{2^j}  \tilde L_{2^{j+1}} \, ,
\end{equation}
where $\tilde M_i = M_i^{(n,h)} /m(nh)$ and $\tilde L_i =  L_i^{(B)}(n,h) / (Bn/h^2)^{1/4}$ are the renormalized weights and E-LPP (we drop the dependence on $n,h,B$ for notational convenience).

As in the proof of Lemma \ref{lem:moment}, we fix some $\gd>0$ such that $1/\ga-1/2 >2\gd$, and as for \eqref{split-sup-unionbound}, the probability in Lemma~\ref{lem:Omeg} is bounded by
\begin{equation}
\label{eq:split-sup-discrete}
\sum_{j=0}^{\log_2 \ell}  \Big[ \bbP\Big( \tilde {L}_{2^{j+1}} > C' \log t \, (2^{j+1})^{1/2+\gd} \Big) + \bbP\Big( \tilde M_{2^j} > C'' \, \frac{t}{\log t} (2^{j})^{-1/\ga+\gd} \Big)  \Big]\, .
\end{equation}

For the first probability in the sum, we obtain from Theorem \ref{thm:generalLPP}-(ii) that
 provided that $C' (\log t)  2^{j\gd}\ge 2 C_0^{1/2}$ 
\begin{equation}
\label{eq:PLtilde}
\bbP\Big(\tilde L_{2^{j+1}} > C' \log t \, (2^{j+1})^{1/2+\gd)} \Big) \le \Big(\frac12\Big)^{C' (\log t)  2^{j\gd}} \le t^{- (\log 2) C' 2^{j\gd}}\, .
\end{equation}
Then, the first sum in \eqref{eq:split-sup-discrete} is bounded by $t^{-a}$
provided that $C'$ had been fixed large enough.

For the second probability in \eqref{eq:split-sup-discrete}, we use Lemma \ref{lem:Mell} above to get that for any $a<\ga$
\begin{align}
\notag
\bbP\Big( \tilde M_{2^j} > C'' \, \frac{t}{\log t} (2^{j})^{-1/\ga+\gd} \Big) 
&\le  \bbP\Big(  M_{2^j}^{(n,h)} > C''' \, \frac{t}{\log t} (2^{j})^{\gd/2} m(nh 2^{-j} ) \Big) \\
&\le  c (\log t)^a t^{-a} (2^{j})^{-a\gd}\,  .\label{eq:PMtilde}
\end{align}
For the first inequality, we used Potter's bound to get that $m(nh 2^{-j} ) \le c m(nh) (2^{j})^{-1/\ga +\gd/2}$.
We conclude that the second sum in \eqref{eq:split-sup-discrete} is  bounded by a constant times $(\log t)^a t^{-a}$.

All together, and possibly decreasing the value a $a$ (by an arbitrarily small anount), this yields Lemma \ref{lem:Omeg}.
\end{proof}

Let us now turn to the case of $T_{n,h}^{\gb_{n,h},(>\ell)}$.
We first need an analogue of Lemma~\ref{lem:Omeg}.
\begin{lemma}
\label{lem:Omegabis}
There exists a constant $c$ such that for any $B\ge 1$, $n,h \in \bbN$ and $0\le \ell \le nh$, for any $t>1$
\[\bbP\Big( \sup_{ \Delta\, : \,\ent(\Delta)  \le B }   \Omega_{n,h}^{(>\ell)} (\Delta)  \ge   t \times m(nh /\ell) \, \ell^{1/2}\big( B n/h^2 \big)^{1/4} \Big) \le c t^{-\ga \ell /3} +  e^{-c \sqrt{t}} \, .\]
\end{lemma}
\begin{proof}
Analogously to \eqref{decomposeOmega}, we get that 
\begin{align}
\label{decomposeOmegabis}
\frac{1}{m(nh/\ell) \ell^{1/2}( Bn/h^2)^{1/4}} & \times \sup_{\Delta: \ent(\Delta) \le B}\Omega_{n,h}^{(>\ell)}(\Delta)  \notag \\
&  \le  \sum_{j=0}^{\log_2 (nh/\ell)} \frac{ M_{2^j \ell}^{(n,h)}}{m(nh/\ell)}  \frac{ L_{2^{j+1} \ell}^{(n,h)} }{\ell^{1/2}( Bn/h^2)^{1/4}} \, ,
\end{align}
Then, we get similarly to \eqref{eq:PLtilde}-\eqref{eq:PMtilde} that for any $\gd>0$:
(a) thanks to Theorem~\ref{thm:generalLPP}-(ii) we have
\begin{equation}
\bbP\Big(  \frac{  L_{2^{j+1} \ell}^{(n,h)} }{\ell^{1/2} (Bn/h^2)^{1/4}} \ge C' \sqrt{t}  (2^{j+1})^{1/2+\gd}\Big) \le  \Big(\frac12 \Big)^{C' \sqrt{t} 2^{j\gd} }  \le e^{ - c\sqrt{t}\, 2^{\gd j}} \, ;
\end{equation}
(b) thanks to Lemma~\ref{lem:Mell} we have
\begin{equation}
\bbP\Big(  \frac{ M_{2^j \ell}^{(n,h)}}{m(nh/\ell)} \ge C'' \sqrt{t} (2^{j})^{-1/\ga +\gd} \Big) \le c  t^{-\ga \ell/3} (2^j)^{-\ga\gd \ell/2} \, . 
\end{equation}
Lemma \ref{lem:Omegabis} follows from a bound analogous to \eqref{eq:split-sup-discrete}.
\end{proof}
Then,  setting $\hat \gb_{\ell} = (\gb_{nh} m(nh/\ell))^{4/3} (\ell^2 n/h^2))^{\frac{1}{3}}$ so that we have $\hat\gb_{\ell} = \gb_{n,h} m(nh/\ell)\ell^{1/2} ( \hat \gb_{\ell} n/h^2)^{1/4}$, we obtain similarly to \eqref{Tnh-large} that
\begin{align*}
\bbP\Big( T_{n,h}&^{\gb_{n,h},(>\ell)}  \ge  b \times \hat\gb_{\ell}  \Big)  \\
& \le   \sum_{k\ge  0 } \bbP \Big(  \gb_{n,h}\sup_{ \Delta\, : \,\ent(\Delta)  \le 2^{k}   b  \hat \gb_{\ell} }   \Omega_{n,h}^{(>\ell)} (\Delta)  \ge 2^{k-1}  b  \hat \gb_{\ell}  \Big) \notag \\
  &   \le   \sum_{k\ge  0 } \bbP\bigg( \sup_{ \Delta\, : \,\ent(\Delta)  \le 2^{k}   b  \hat \gb_{\ell} }   \Omega_{n,h}^{(>\ell)} (\Delta)  \ge 2^{k-1}  b \,  \,  m(nh/\ell) (\ell^2 \hat \gb_{\ell}  n/h^2 )^{1/4}  \bigg)\\
 & \le  \sum_{k\ge  0 }  \Big( c(2^k b)^{- \ga \ell /4} + e^{- c 2^{3k/8} b^{3/8}} \Big) \le c' b^{-\ga \ell/4} + e^{-c' b^{1/4}}
  \, .
\end{align*}
This concludes the proof of Proposition \ref{thm:discreteELPP}.

\section*{Acknowledgements}
We are most grateful to N. Zygouras for many enlightening discussions.

\bibliographystyle{plain}
\bibliography{biblioHTBN.bib}

\end{document}